\documentclass{amsart}
\usepackage{amsmath}
\usepackage{amssymb}
\usepackage{amsthm}
\usepackage{amsfonts}
\usepackage{url}
\usepackage{hyperref}
\usepackage{appendix}
\usepackage{verbatim}
\usepackage{graphicx,color}
\usepackage{lineno}

\usepackage[curve,tips]{xypic}
\SelectTips{eu}{11}
\UseTips

\usepackage{tikz}
\usepackage{pst-knot}

\newtheorem{thmA}{Theorem} 
 
\newtheorem{propA}[thmA]{Proposition}

\title{Profinite rigidity, Kleinian groups, and the cofinite Hopf property}
\author{M. R. Bridson}
\author{A. W. Reid}
\address{\newline Mathematical Institute, 
\newline Andrew Wiles Building,
\newline University of Oxford,
\newline Oxford OX2 6GG, UK}
\email{ bridson@maths.ox.ac.uk}
\address{\newline Department of Mathematics,
\newline Rice University, 
\newline Houston, TX 77005, USA}
\email{ alan.reid@rice.edu}

\def\vv{\Pi} 

\def\a{\alpha}
\def\b{\beta}

\def\-{\overline}

\def\wh{\widehat}

\def\G{\Gamma}

\def\ker{{\rm{ker}}\ }

 \def\H{\mathbb{H}}
 \def\Z{\mathbb{Z}}
 \def\R{\mathbb{R}}
 \def\Q{\mathbb{Q}} 
  
 \def\C{\mathbb{C}} 
 \def\D{\Delta}

\def\ab{\rm{ab}}

\DeclareMathOperator{\SL}{SL} \DeclareMathOperator{\PSL}{PSL}
\DeclareMathOperator{\GL}{GL} 
 
 \DeclareMathOperator{\SO}{SO}

\def\Vol{\rm{Vol}}

\def\qed{ $\sqcup\!\!\!\!\sqcap$}

\def\tr{\mbox{\rm{tr}}\, }
\def\P{\mbox{\rm{P}}}

\def\Or{\mbox{\rm{O}}}
\def\G{\Gamma}

\def\<{\langle}
\def\>{\rangle}

\def\G{\Gamma}
\def\wh{\widehat}

\def\onto{\twoheadrightarrow}

\newtheorem{theorem}{Theorem}[section]
\newtheorem{lemma}[theorem]{Lemma}
\newtheorem{corollary}[theorem]{Corollary}
\newtheorem{proposition}[theorem]{Proposition}

\theoremstyle{definition} 

\newtheorem{remark}[theorem]{Remark}

\begin{document}
 
\date{}

\begin{abstract}  
Let $\G$ be a non-elementary Kleinian group and $H<\G$ a finitely generated, proper subgroup. 
We prove that if $\G$ has finite co-volume, then the 
profinite completions of $H$ and $\G$ are not isomorphic. 
If $H$ has finite index in $\G$, then there is a finite group onto which $H$ maps but $\G$ does not. 
These results
streamline the existing proofs that there exist full-sized groups that are profinitely rigid in the absolute sense. They build on a circle
of ideas that can be used to distinguish among the profinite completions of subgroups of finite
index in other contexts, e.g. limit groups.  
We construct new examples of profinitely rigid groups,
including the fundamental group of the hyperbolic $3$-manifold $\Vol(3)$ 
and of the $4$-fold cyclic branched cover of the figure-eight knot. We also prove that if a lattice in ${\rm{PSL}}(2,\C)$ 
is profinitely rigid, then so is its normalizer in $\PSL(2,\C)$.  
\end{abstract}




\maketitle

\centerline{\em{Dedicated to Gopal Prasad on the occasion of his $75$th birthday}}

%
%
%
%

\section{Introduction}\label{intro}

{\em Rigidity theorems} are propositions that tell us objects of a certain kind are equivalent in a weak sense only if they
are equivalent in an apparently-stronger sense. A celebrated prototype for such theorems comes from the work of
Mostow \cite{Mo} and Prasad \cite{Pr}: beyond dimension 2, if complete hyperbolic manifolds of finite volume have
isomorphic fundamental groups, then the manifolds must be isometric. In our previous work \cite{BR1, BCR, BRW, BMRS1, BMRS2} we showed
that hyperbolic orbifolds also play an important role in the study of profinite rigidity. 

Profinite rigidity theorems are propositions that tell us
groups of a certain kind will have the same set of finite quotients only if the groups are isomorphic. A finitely generated, residually finite
group $\G$  is  {\em profinitely rigid} in the absolute sense if its set of finite quotients distinguishes it from all other finitely generated,  residually finite groups: in the language of profinite completions (recalled in Section \ref{s2}), 
$\wh{\Lambda}\cong\wh{\G}\implies \Lambda\cong\G$.

In our work with McReynolds and Spitler \cite{BMRS1, BMRS2}, we constructed arithmetic lattices  in ${\rm{PSL}}(2,\R)$
and ${\rm{PSL}}(2,\C)$ that are profinitely rigid in the absolute sense. It remains unknown whether all such lattices are
profinitely rigid. The proofs in \cite{BMRS1, BMRS2} have two stages: first, given the lattice $\G$ and a finitely generated, residually finite
group $\Lambda$ with $\wh{\Lambda}\cong\wh{\G}$, one tries to construct a Zariski-dense representation 
$\Lambda\to {\rm{PSL}}(2,\C)$ with image  in $\G$; one then has to argue that if $\Lambda\to\G$ were not  surjective, a contradiction
would ensue. Our main objective in the current article is to develop a general framework that, in particular, provides a 
more uniform and comprehensive treatment of the second stage in these proofs.

\begin{thmA}\label{t1}
\label{main1}
Let $\G$ be a Kleinian group of finite co-volume. If $H<\G$ is a finitely generated, proper subgroup, then  $\wh{\G}$ and $\wh{H}$
are not isomorphic.\end{thmA}

The most interesting part of the proof of this theorem concerns subgroups of finite index, where our arguments apply to 
finitely generated, non-elementary Kleinian groups in general.

\begin{thmA}\label{t2}
\label{morefinite1}
Let $\G$ be a finitely generated, non-elementary Kleinian group. For each proper subgroup of finite index $H<\G$, there exists a finite
group $Q$ such that $H$ maps onto $Q$ but $\G$ does not. 
\end{thmA}

Our proof of Theorem \ref{t2}  relies on specific properties of Kleinian groups, including
Liu's recent result \cite{Liu} that if $M$ is a complete, orientable hyperbolic $3$-manifold of finite volume,
then there are only finitely many other such manifolds $N$ with $\wh{\pi_1(N)}\cong\wh{\pi_1(M)}$. 
But the  general outline of our proof is built around
general facts concerning (virtual) epimorphisms of finitely generated profinite groups that have wider implications for the
study of profinite rigidity; see in particular Propositions \ref{p1} and \ref{p2}. The applications in this article are based on the
following consequence of these results.

\begin{propA}\label{p-intro} Let $G$ be a finitely generated, residually finite group and $H<G$ a subgroup of finite index.
If every finite quotient of $H$ is also a quotient of $G$, then there is a
continuous surjection $\wh{G}\onto\wh{H}$ with finite kernel.
\end{propA} 

This proposition is most useful in geometric situations where one can control finite normal subgroups and one has non-zero invariants
(for example $\ell_2$-betti numbers)  that
increase on passage to finite-index subgroups. In such circumstances, one can often prove that the groups concerned are
{\em cofinitely Hopfian}, meaning that every endomorphism $\phi:G\to G$ with $[G:\phi(G)]<\infty$ must be an automorphism \cite{BGHM}. 
We illustrate the potential of this circle of ideas by proving that every proper, finite-index subgroup of a non-abelian limit
group has a finite quotient that the ambient group does not have (Proposition \ref{p:limit}). We  also explain, in Section \ref{s:go-up}, how
these general observations, combined with Mostow-Prasad rigidity, allow one to extend profinite rigidity results for lattices to
over-lattices. For example:

\begin{thmA}\label{t:intro-norm}
If a lattice $\G<{\rm{PSL}}(2,\C)$ is profinitely rigid, then so is its normaliser in $\PSL(2,\C)$. 
\end{thmA}

We have discussed how Theorems \ref{t1} and \ref{t2} simplify the latter stages of the proofs in \cite{BMRS1}. This 
simplification provides us with encouragement to  
search for further examples of profinitely rigid lattices $\G<{\rm{PSL}}(2,\C)$. We take up this challenge in the final sections of this paper,
where we study cyclic branched covers of the figure-eight knot and related families of groups. In particular we prove the following
theorem. 
Here, $\Vol(3)$ is a particular closed orientable hyperbolic 3-manifold whose volume is that of the regular ideal tetrahedron in $\H^3$;
it has  recently been established  \cite{GHMTY} that $\Vol(3)$ is the unique closed orientable
hyperbolic $3$-manifold with this volume, which is the third-smallest possible volume for closed manifolds.

\begin{thmA}\label{t:4-fold}
The fundamental group of the 4-fold cyclic branched cover of the figure-eight knot  is profinitely rigid, and so 
is the fundamental group of $\Vol(3)$.
\end{thmA}

In order to illustrate the practical nature of the results in Section \ref{s:go-up}, we shall explain how profinite rigidity extends
to lattices containing those in  Theorem \ref{t:4-fold}; see Corollaries \ref{c:up-from-G4} and \ref{c:Omega4}.
Our results in Section \ref{s:galois} also establish {\em Galois rigidity} for a class of groups that includes these. Galois rigidity
(whose definition we recall in Section \ref{s:basic}) provides tight control on the $(\P){\rm{SL}_2}$-character variety of the groups concerned
and plays a crucial role in the first stage of the proofs in \cite{BMRS1, BMRS2}. 

A further notion of rigidity relevant to Theorem \ref{main1} is {\em Grothendieck rigidity}. 
In \cite{groth}, Grothendieck  initiated the study of isomorphisms
$\wh{u}:\wh{H}\overset{\simeq}\longrightarrow\wh{\G}$ induced by inclusions of discrete groups $u:H\hookrightarrow\G$. One calls
$(\G,H)_u$  a {\em Grothendieck pair} if $\G$ is finitely generated and residually finite and $\wh{u}$ is an isomorphism
but $u$ is not.  If there is no Grothendieck pair $(\G,H)_u$ with $H$ finitely generated, then
$\G$ is  said to be Grothendieck rigid. 
Bridson and Grunewald \cite{BG}
proved that there exist Grothendieck pairs of finitely presented groups. The analogous problem for finitely generated groups was settled earlier by Platonov and Tavgen \cite{PT}. Long and Reid  \cite{LR} established Grothendieck 
rigidity for  the  fundamental groups of all closed geometric $3$-manifolds and finite volume hyperbolic $3$-manifolds.
Theorem \ref{main1} proves something stronger in the hyperbolic case, ruling out abstract isomorphisms 
$\wh{H}\cong\wh{\G}$, not just those induced by inclusions $H\hookrightarrow\G$. The analogous stronger
statement is false in the non-hyperbolic setting \cite{hempel, funar}.

An alternative approach to Grothendieck rigidity proceeds via subgroup separability.
If $(\G,H)_u$ is a Grothendieck pair, then $H$ is dense in the profinite topology on $\G$. Thus any residually finite group that is LERF (all finitely generated subgroups are closed in the profinite topology) will be Grothendieck rigid. This observation was applied in  \cite{PT} to free groups and Fuchsian groups.  Following the work of  \cite{Ag} and \cite{Wi} one knows that all finitely generated Kleinian groups are LERF, so they
too are Grothendieck rigid.  
Recently, Sun proved that the fundamental group of every compact $3$-manifold is Grothendieck rigid \cite{Sun}. 

This paper is organised as follows. In Section 2 we gather the definitions and background material needed
in the sequel. In Section 3 we present two basic results about (virtual) surjections of profinite groups, Propositions \ref{p1} and \ref{p2}, and
illustrate their utility with some immediate applications.  In Section 4 we delve further into the significance of finite normal subgroups
in profinite completions and discuss results about extending profinite rigidity to over-lattices, including Theorem \ref{t:intro-norm}.
Section 5 contains the remainder of the proof of Theorems \ref{t1} and \ref{t2}. In Section 6 we construct a new (finite) family of
profinitely rigid lattices in ${\rm{PSL}}(2,\C)$, starting with Theorem \ref{t:4-fold}; Section 7 contains an explanation of the Mathematica calculations used to compute the character variety of the fundamental group of the $4$-fold cyclic branched cover of the figure-eight knot. 

\bigskip
  
\noindent{\bf Acknowedgements:}  We thank Nikolay Nikolov for helpful conversations in connection with Proposition \ref{p1}. We also  thank our co-authors from  \cite{BMRS1} and \cite{BMRS2}, D.~B.~McReynolds and R. Spitler,  
for many helpful discussions concerning profinite rigidity for Kleinian groups. 
While carrying out this work, we benefited from the hospitality of the Hausdorff Institute for Mathematics (Bonn), the Institute for Advanced Study (Princeton), ICMAT (Madrid) and the University of Auckland;
we thank them all. We also wish to thank the referee for some useful suggestions.  
The second author acknowledges with gratitude the financial support of the National Science Foundation.

\section{Preliminaries}\label{s2}

We shall assume that the reader is familiar with basic facts about profinite groups; the book
\cite{RZ} is a standard reference.  By definition, the {\em profinite completion} $\wh{\G}$ of a group $\G$
is the limit of the inverse system consisting of the finite quotients $\G/N$ and the natural maps $\G/N\onto \G/M$
for $N<M$. Beyond the basic theory, we shall need the
Nikolov-Segal Theorem \cite{NS} that every subgroup of finite index in a finitely
generated profinite group is open.
We shall make frequent use of the following basic result describing the correspondence between the
finite-index  subgroups of a discrete group and those of  its profinite completion
(see \cite[Proposition 3.2.2]{RZ}, and note that we have
used \cite{NS} to replace ``open'' by ``finite index'').\\[\baselineskip]
\noindent{\bf{Notation.}}  Given a subset $X$ of a profinite group
$G$, we write $\overline X$ to denote the closure of $X$ in $G$.

\begin{proposition}
\label{correspondence}
If $\G$ is a finitely generated, residually finite group, then
there is a one-to-one correspondence between the
set $\mathcal{X}$ of subgroups of finite index in $\G$ and the set $\mathcal{Y}$ of subgroups of finite index in $\widehat{\G}$.
Identifying $\G$ with its image in $\widehat{\G}$, this correspondence
is given by:

\begin{itemize}
\item For $H\in {\mathcal X}$, $H \mapsto \overline{H}$.
\item For $Y\in {\mathcal Y}$, $Y\mapsto Y \cap \G$.
\end{itemize}

\noindent If $H,K\in {\mathcal X}$ and $K<H$ then
$[H : K] = [\overline{H}:\overline{K}]$. Moreover, $K\triangleleft H$
if and only if $\overline{K} \triangleleft \overline{H}$,
and in this case $\overline{H}/\overline{K} \cong H/K$.
\end{proposition}

The last clause implies that $\G$ and $\wh{\G}$ have the same finite quotients.
A useful refinement of this fact is the observation that the set of finite quotients of $\G$ determines $\widehat{\G}$ up to
isomorphism -- see \cite[Theorem 3.2.7]{RZ}), for example, which we repeat here
using \cite{NS} to replace {\em topological isomorphism} with {\em isomorphism} (as abstract groups).
We write $\mathcal{C}(\Gamma)$ to denote the set of isomorphism classes of finite groups onto which 
$\G$ surjects.

\begin{theorem}
\label{allfinite}
Suppose that $\Gamma_1$ and $\Gamma_2$ are finitely generated abstract
groups. Then $\widehat{\Gamma}_1$ and $\widehat{\Gamma}_2$ are isomorphic if
and only if $\mathcal{C}(\Gamma_1) = \mathcal{C}(\Gamma_2)$.
\end{theorem}

\subsection{Kleinian groups and Mostow-Prasad rigidity}  

A {\em Kleinian group} is, by definition, a discrete subgroup of ${\rm{PSL}}(2,\C)={\rm{Isom}}^+(\mathbb{H}^3)$. 
We shall be concerned only with
finitely generated Kleinian groups.  
A finitely generated Kleinian group is   {\em elementary} if it contains an abelian
subgroup of finite index; otherwise it is {\em non-elementary}.
The {\em co-volume} of a Kleinian group $\G$ is the Riemannian volume of the quotient
orbifold ${\mathbb{H}}^3/\G$. Thus Kleinian groups of finite co-volume are precisely the lattices in ${\rm{PSL}}(2,\C)$.

Mostow \cite{Mo} (for uniform lattices) and Prasad \cite{Pr} (covering the non-uniform case) proved that if $G$ is
a semisimple Lie group with trivial  center and no compact factors, and $G$ is not locally isomorphic to $\PSL(2,\R)$,  
then any pair of lattices $\G_1, \G_2<G$ that are  isomorphic as abstract groups must be conjugate in $G$. We shall make frequent use
of this fact in the case $G={\rm{PSL}}(2,\C)$.
 
\subsection{First $\ell_2$-betti numbers} 
The first $\ell_2$-betti number of a finitely presented group $\G$ can be calculated (or defined) using L\"{u}ck approximation:
$b_1^{(2)}(\G)=\lim_n b_1(H_n)/[\G:H_n]$, where 
$b_1(H)$ is the standard first betti number and $(H_n)$
is a nested sequence of finite-index normal subgroups of $\G$ intersecting in the identity. The first assertion in 
the following lemma is immediate from this description and the second is \cite[Corollary 2.11]{BCR}.

\begin{lemma}\label{l2}
Let $\G_1$ and $\G_2$ be finitely presented groups.
\begin{enumerate}
\item If $H<\G_1$ is a subgroup of finite index $d$, then  $b_1^{(2)}(H) = d\ b_1^{(2)}(\G)$.
\item If $\wh{\G}_1 \cong \wh{\G}_2$ then $b_1^{(2)}(\G_1) = b_1^{(2)}(\G_2)$.
\end{enumerate}
\end{lemma}

\subsection{Goodness and torsion} 
Following Serre \cite{Se}, one says that a group $\G$ is {\em good} 
if for every finite $\G$-module $M$, the homomorphism of cohomology groups
$H^n(\widehat{\G};M)\rightarrow H^n(\G;M)$
induced by the natural map $\G\rightarrow \widehat{\G}$ is an isomorphism
between the cohomology of $\G$ and the continuous cohomology of $\widehat{\G}$. 
(See \cite{Se} and \cite[Chapter 6]{RZ} for details about the cohomology of profinite groups.) 

Fuchsian groups are good \cite{GJZ}. Serre  \cite[Chapter 1, Section 2.6]{Se} gives a criterion for
when extensions of good groups by good groups are good; by applying this, one sees that
the fundamental groups of surface bundles over the circle are good. Goodness is an invariant of
commensurability, so it follows from Agol's Virtual Fibering Theorem \cite{Ag} that all lattices in ${\rm{SL}}(2,\C)$ and
${\rm{PSL}}(2,\C)$ are good. More generally, all Kleinian groups are good; see \cite[p.102]{AFW}.
 The most important consequence of this for us in the present paper
comes from the following well-known fact (which is \cite[Corollary 2.16]{BCR}).

\begin{lemma}\label{l:good-tf}
If a finitely generated, residually finite group $G$ is good and has finite cohomological dimension,
then $\widehat{G}$ is torsion-free.
\end{lemma}

\begin{corollary}\label{c:good-SL}
If $\G$ is a finitely generated, discrete torsion-free subgroup of  ${\rm{PSL}}(2,\C)$, then $\wh{\G}$ is torsion-free.
\end{corollary}

We shall also need a variant of Lemma \ref{l:good-tf} that allows for limited torsion.
The following is a weak form of a result of Minasyan and Zalesskii \cite[Corollary 3.5]{MZ}.

\begin{proposition}\label{p:no-normal} Let $\G$ be a finitely generated, residually finite 
group that has a subgroup of finite index  with finite cohomological dimension. If $\G$ is good, then every element
$g\in\wh{\G}$ of prime order is conjugate into $\G$.
\end{proposition} 
 
\begin{corollary}\label{c:SL-no-N}
If $\G<{\rm{PSL}}(2,\C)$ is a finitely generated non-elementary Kleinian group, then $\wh{\G}$ does not contain a 
non-trivial, finite normal subgroup.
\end{corollary}

\begin{proof}
A non-trivial finite normal subgroup $N<\wh{\G}$ would contain an element $g$ of prime order. Proposition \ref{p:no-normal}
tells us that the conjugacy class of $g$, which is contained in $N$, must intersect $\G$. Thus $H=N\cap\G$ is a non-trivial finite
normal subgroup of $\G$. But $\G$ does not contain such subgroups, because if it did then the non-empty, totally-geodesic 
submanifold ${\rm{Fix}}(H)\subset\H^3$ would be invariant under $\G$; this submanifold cannot be 2-dimensional, because 
$\G$ (hence $H$) acts effectively and does not contain orientation-reversing isometries; and if ${\rm{Fix}}(H)$ were a line
or a point, then the discrete action of $\G$ on it would imply that $\G$ was virtually cyclic, hence elementary.
\end{proof}

\begin{remark}
There is an alternative proof of Corollary \ref{c:SL-no-N} wherein one argues
 as in the proof of Theorem 1.2 of \cite{LLR} that for infinitely many finite fields $\mathbb{F}$, $\G$ maps onto the
 finite simple group ${\PSL}(2,\mathbb{F})$ with torsion-free kernel. As these groups are simple
 and unbounded in size, for infinitely many of the induced maps $\wh{\G}\to {\PSL}(2,\mathbb{F})$, any
 finite normal subgroup $N<\wh{\G}$ would be in the kernel. But this would contradict Corollary \ref{c:good-SL},
 because the kernel of $\wh{\G}\to {\PSL}(2,\mathbb{F})$ is the completion of the kernel of $\G\to {\PSL}(2,\mathbb{F})$, which is torsion-free.
 \end{remark}
 
\subsection{Trace fields, character varieties, and Galois rigidity} 
In order to fully understand the construction of profinitely rigid groups in the penultimate
 section of this paper, the reader will need to be familiar with
the basic arithmetic theory of Kleinian groups, as described in \cite{MR}, and to have some familiarity with the main 
ideas in our papers with McReynolds and Spitler \cite{BMRS1,BMRS2}, where we proved
that certain Kleinian and Fuchsian groups are profinitely rigid. For the convenience of the reader,
we recall the basic terminology here. 

Let $\G$ be a finitely generated subgroup of $\PSL(2,\mathbb{C})$ and  let  $\G_1$ be its pre-image  in $\SL(2,\mathbb{C})$.
It will be convenient to say $\G$ is {\em Zariski dense} in $\PSL(2,\C)$ when what we actually mean is that $\G_1$ is Zariski dense in $\SL(2,\C)$. 
The \textit{trace-field} of $\G$ is defined to be the field 
\[ K_\G=\mathbb{Q}(\mathrm{tr}(\gamma)~\colon~ \gamma \in \G_1). \] 
The algebra consisting of finite $K_\G$-linear combinations
of the matrices $\gamma\in\G_1$  is a quaternion algebra, denoted $A_0\G$; see \cite[Chap 3]{MR}.  
If $\G$ is a lattice, then $K_\G$ is a number field, of degree $n_K$ say. 
 If $K_\G=\Q(\theta)$ for some algebraic number $\theta$, then the Galois conjugates of $\theta$, say $\theta=\theta_1,\dots,\theta_{n_K}$ provide embeddings $\sigma_i\colon K_\G\to\C$ defined by $\theta\mapsto\theta_i$.  These in turn can be used to build $n_K$ Zariski dense non-conjugate representations $\rho_{\sigma_i}\colon \G \to \PSL(2,\C)$ with the property that $\tr(\rho_{\sigma_i}(\gamma))=\sigma_i(\tr\rho(\gamma))$ for all $\gamma\in \G$. These are the {\em Galois conjugates} of the inclusion map $\G\hookrightarrow \mathrm{PSL}(2,\mathbb{C})$. 

As in \cite{BMRS1} we will be most interested in groups $\G$ with the fewest possible Zariski dense representations.  
We say that $\G$  is {\em Galois rigid} if, up to conjugacy, the only Zariski dense representations of $\G$ in ${\rm{PSL}}(2,\C)$
are the Galois conjugates of the inclusion.  (In \cite{BMRS1} 
we defined Galois rigidity for more general finitely generated groups, but are interests here are narrower.) 
The crucial results relating Galois rigidity to profinite rigidity are Theorem 4.8 and Corollaries 4.10 and 4.11 of  \cite{BMRS1}.

\section{Two basic results}\label{s:basic}

In this section we present two general results concerning (virtual) epimorphisms of profinite groups. 
The proofs that we shall present for the theorems described in the introduction  
will illustrate the utility of these results. 
We expect these results to have many further applications in the
context of profinite rigidity.

\subsection{Profinite surjections}\label{prof_surject}

We begin with a one-sided version of Theorem \ref{allfinite}. 

\begin{proposition}\label{p1} If $G$ and $H$ are finitely generated profinite groups
and every finite quotient of $H$ is also a quotient of $G$, then there is a
continuous surjection $\wh{G}\onto\wh{H}$. 
\end{proposition}
 
\begin{proof} The proposition is trivial if $H$ is finite, so suppose $H$ is infinite.
For each positive integer $n$ we denote by $I_n$ the intersection of all subgroups in $H$ that have index at most $n$,
and we consider the quotient 
$$H(n) := H/I_n.$$
There is an obvious surjection $H(n+1)\onto H(n)$ and $H$ is the
inverse limit of the system $\{H(n+1)\onto H(n)\}_n$.

Let $T(n)$ be the set of (necessarily open) normal subgroups $L\le G$ such that 
$G/L\cong H(n)$. Let $T=\bigcup_n T(n)$.
Note that $T(1)=G$ and $T(n)$ is non-empty since $H(n)$ is a quotient of $G$.
Each $L\in T({n})$ is contained in some  $L' \in T(n-1)$, namely the kernel of 
the composition $G\to G/L\cong H(n)\to H(n-1)$. We say 
$L$ is a descendant of $L'$ and draw an arrow from $L'$ to $L$. This defines
a connected, locally-finite directed graph with root $G$ and vertex set $T$. 
Let $G=L_1 > L_2>L_3>\dots$ be an infinite directed path in this graph
with $G/L_n \cong H(n)$  for all $n$. The quotient
of $G$ by the closed subgroup $N= \bigcap_n L_n$
is isomorphic to the inverse limit of the $H(n)$,
which is $H$.
\end{proof}

\begin{remark} 
Theorem \ref{allfinite} follows easily from Proposition \ref{p1} and the fact that finitely generated profinite groups
are Hopfian, i.e. every continuous surjection $G\to G$ is an isomorphism (as explained in the 
last paragraph of the proof of Proposition \ref{p2} below).
\end{remark}

Proposition \ref{p1} is useful in much the same way as Theorem \ref{allfinite},
particularly in combination with Proposition \ref{correspondence}.

\subsection{The cofinite Hopf property}\label{cofiniteHopf}
 
The second general tool that we gather comes from the observation 
that Hirshon's proof \cite{hirshon} of the discrete analogue of the following proposition carries over to the profinite setting.

\begin{proposition}\label{p2} If $G$ is a finitely generated profinite group and  $H<G$ is a subgroup of finite index, then every epimorphism $\alpha: G\to H$ has finite kernel; in particular, if $G$ has no finite normal subgroup then $\alpha$ is injective.
\end{proposition}
 
\begin{proof}
Let $\alpha : G\onto H$ be an epimorphism, define 
$H_n := \alpha^n(G)$, let $\theta_n : H_n\onto H_{n+1}$ be the restriction of $\alpha$ and let $K_n=\ker \theta_n$. 
Note that $H_{n+1} < H_n$ and that $H_n$ has finite index in $G$. Also, $K_{n+1} < K_n = H_n\cap \ker \alpha$.

Let $N_n$ be the core of $H_{n+1}$ in $H_n$, that is, the intersection of the conjugates $g^{-1}H_{n+1}g$ with $g\in H_n$. Then $\theta_n(N_n) 
\subseteq N_{n+1}$, so $\theta_n$
induces an epimorphism $\overline\theta_n: H_n/N_n\onto H_{n+1}/N_{n+1}$. As the groups $H_i/N_i$ are finite, for sufficiently large $n$, say $n\ge n_0$, the map
$\overline\theta_n$ must be an isomorphism.  This forces $K_n < N_n$. But $N_n < H_{n+1}$, and by definition $K_i = H_i\cap \ker \alpha$,
 so $K_n = K_{n+1}$ for all $n\ge n_0$.

In order to obtain a contradiction, we assume that $\ker \alpha$ is infinite. In this case, $K_{n_0} = H_{n_0}\cap \ker\alpha$ is infinite, in particular non-trivial. Let $k$ a non-trivial element of $K_{n_0}$. As
$K_{n_0} = K_n < H_n$ for all $n\ge n_0$, there exist elements $k_n\in G$ such that $\alpha^n(k_n)=k$. Note that $\alpha^{n+1}(k_n) = \alpha(k) = 1$.

As $G$ is residually finite, there is a finite quotient $\pi: G\to Q$ such that $\pi(k)\neq 1$. The contradiction that we seek is obtained by noting that the
maps $p_n = \pi\circ \alpha^n : G\to Q$ are all distinct, because $p_n(k_n) \neq 1$ but $p_m(k_n)=1$ for all $m>n$; this is nonsense because there are only finitely many
maps from a finitely generated profinite group to a finite group. 
\end{proof}

\begin{remark}\label{rem_iso} The above argument shows that 
$\theta_n$ is an isomorphism for sufficiently large $n$, regardless of whether $G$ contains finite normal subgroups.
\end{remark}

The following result is an immediate
consequence of Propositions \ref{p1} and \ref{p2}.  Recall that, given a class
of groups $\mathcal G$, a group $\G\in \mathcal G$ is said to be {\em profinitely rigid in $\mathcal G$} if 
$\widehat{\Lambda}\cong \widehat{\G}$  implies $\Lambda\cong\G$ provided $\Lambda\in\mathcal G$.

\begin{corollary}\label{c2} Let $\mathcal G$ be a class of finitely generated
groups that is closed under passage to subgroups of finite index.
Suppose that $\G\in\mathcal{G}$ is profinitely rigid in $\mathcal G$
and that $\wh{\G}$ does not contain a non-trivial finite normal subgroup.
If $H<\G$ is a proper subgroup of finite index, then either $H\cong \G$ or else $H$ has a finite quotient that $\G$ does not have.
\end{corollary} 

The results of \cite{BCR} imply that Corollary \ref{c2} applies to any non-elementary Fuchsian group $\G$
(cf.~\cite[Corollary 3.7]{BMRS2}). It also applies when $\mathcal{G}$ is the class of 3-manifold groups and $\G$ 
is the fundamental group of any once-punctured torus bundle over the circle \cite{BRW}.

\subsection{Fuchsian groups and limit groups: $b_1^{(2)}>0$ and goodness}

In order to apply Corollary \ref{c2} to non-elementary Fuchsian groups we had to quote one of the main results
of \cite{BCR} --  non-isomorphic Fuchsian groups can be distinguished from each other by their finite quotients. A more useful
observation is that the proofs of profinite rigidity in \cite{BCR} and \cite{BMRS2} can be shortened considerably
by invoking the following consequence of Propositions \ref{p1} and \ref{p2}. 

\begin{proposition}\label{p4}
If $H$ is a proper subgroup of finite index in a finitely generated non-elementary Fuchsian group $\G$, then $H$ has a finite quotient that $\G$ does not have.
\end{proposition}

\begin{proof} If the proposition were to fail for a finite-index subgroup $H<\G$, then Proposition \ref{p1} would provide 
a continuous epimorphism $\wh{\G}\onto\wh{H}$. As $\G$ is a non-elementary Fuchsian group, $\wh{\G}$ does
not contain a non-trivial finite normal subgroup \cite[Corollary 5.2]{BCR}, so  Proposition \ref{p2} implies that $\wh{\G}\onto\wh{H}$
is an isomorphism. On the other hand, since $b_1^{(2)}(\G)>0$ for non-elementary Fuchsian groups (see for example \cite[Proposition 3.5]{BCR})
and since Fuchsian groups are finitely presented, Lemma \ref{l2} 
tells us that $\wh{\G}\not\cong\wh{H}$ if $H<\G$ is a proper subgroup of finite index -- a contradiction.
\end{proof}

The proof of Proposition \ref{p4} serves as a template for other classes of groups in which (1) one can control finite normal subgroups
in $\wh{\G}$ and (2) one has a profinite invariant (for example $b_1^{(2)}(\G)$) that is non-zero and
increases when one passes to finite-index subgroups (cf.~\cite{BGHM}). 
Goodness provides a means of controlling finite subgroups (Lemma \ref{l:good-tf}
and Proposition \ref{p:no-normal} ), so it is natural
to look for further applications in settings where one knows the groups are good.  
Recall that a finitely generated group $\G$ is called a {\em limit group} if it is fully residually free, i.e.~for every finite subset $X \subset \G$,  there is a homomorphism from $\G$ to a free group that restricts to an injection on $X$.

\begin{proposition}\label{p:limit} Let $\G$ be a non-abelian limit group.
If $H$ is a proper subgroup of finite index in $\G$, then $H$ has a finite quotient that $\G$ does not have.
\end{proposition}

\begin{proof} We follow the proof of Proposition \ref{p4}. 
Limit groups have finite classifying spaces and are good, by  \cite{GJZ}, so $\wh{\G}$ is torsion-free by Lemma \ref{l:good-tf}.
Thus if the proposition were to fail for $H<\G$, then by combining Propositions \ref{p1} and \ref{p2} we would get
an isomorphism $\wh{\G}\cong\wh{H}$. 
But  $b_1^{(2)}(\G)>0$ for non-abelian limit groups, by \cite{BK}. Lemma \ref{l2}
provides us with the desired contradiction, as before. 
\end{proof}
 
A further application of the same argument yields a special case of Theorem \ref{t2}, for which we need the following lemma.

\begin{lemma}
\label{l:l2}
Let $\G$ be a finitely generated non-elementary Kleinian group.  Then $b_1^{(2)}(\G)=0$ if and only if $\G$ has finite co-volume. 
\end{lemma}

\begin{proof} When $\G$ is torsion-free, this  is proved in \cite{LL}. The general case follows from
Lemma \ref{l2}, since finitely generated Kleinian groups are finitely presented.
\end{proof}

\begin{proposition}\label{p:inf-vol} Let $\G$ be a finitely generated non-elementary Kleinian group that  
has infinite co-volume. If $H<\G$ is a proper subgroup of finite index, then $H$ has a finite quotient that $\G$ does not have.
\end{proposition}

\begin{proof} Corollary \ref{c:SL-no-N} assures us $\G$ has no non-trivial finite normal subgroup,
and Lemma \ref{l:l2} tells us $b_1^{(2)}(\G)>0$.
\end{proof}

\section{Profinite rigidity of over-lattices} \label{s:go-up}

Before turning to the proof of Theorems \ref{t1} and \ref{t2}, we elaborate further on the significance of finite normal
subgroups in profinite completions, which emerged in the previous section. By focusing on this, we shall see in particular that the
normaliser of any profinitely rigid lattice  $\G<{\rm{(P)SL}}(2,\C)$ is itself profinitely rigid. 
Mostow-Prasad rigidity plays an important role in these arguments.

\begin{proposition}\label{p:latt-test}
Let $\G$ be a lattice in $\PSL(2,\C)$ that is profinitely rigid (in the absolute sense). 
Let $\Lambda$ be a finitely generated, residually finite group that has a subgroup of  finite index $\Delta$ 
with $\wh{\Delta}\cong \wh{\G}$.
Then $\Lambda$ is isomorphic to a lattice in $\PSL(2,\C)$ if and only if $\widehat{\Lambda}$ does not contain a 
non-trivial finite normal subgroup.
\end{proposition}

\begin{proof}
If $\Lambda$ is Kleinian, then Corollary \ref{c:SL-no-N}  assures us there is no non-trivial finite normal subgroup in $\widehat{\Lambda}$.
The profinite rigidity of $\G$ implies $\Delta\cong \G$. Thus $\Lambda$ is commensurable to a  lattice in $\PSL(2,\C)$, 
and by Mostow-Prasad Rigidity there is a short exact sequence
$$
1\to N\to \Lambda \to \Omega \to 1
$$
with  $N$ is finite and $\Omega$ a lattice in $\PSL(2,\C)$ that contains $\G$. 
If $\Lambda$ is not itself a lattice, then $N\neq 1$ is the  non-trivial finite normal subgroup in $\widehat{\Lambda}$ 
we seek.  
\end{proof}

\begin{corollary} \label{cor-general}
Let $\Lambda_1$ and $\Lambda_2$ 
be finitely generated, residually finite groups with 
$\widehat{\Lambda}_1\cong\widehat{\Lambda}_2$.  Let $\G< \PSL(2,\C)$ be a lattice that is profinitely rigid. 
If $\Lambda_1$ has  a finite-index subgroup  isomorphic to $\G$, then so does $\Lambda_2$. In this case, either
$\Lambda_1$ and $\Lambda_2$ are (isomorphic to)
 commensurable lattices in $\PSL(2,\C)$ with the same covolume,  or else neither is a lattice.
 \end{corollary}
 
 \begin{proof}
 Suppose $\Lambda_1$ contains $\G$ as a subgroup of index $d$. Proposition \ref{correspondence} implies
that $\Lambda_2$ contains a subgroup $\Delta$ of index $d$
with $\widehat{\Delta}\cong\widehat{\G}$. As $\G$ is profinitely rigid, $\Delta\cong\G$. So if $\Lambda_1$
and $\Lambda_2$ are lattices, then they are commensurable and the covolume of each is $1/d$ that of $\G$.
And Proposition \ref{p:latt-test} tells us they are lattices if and only if $\widehat{\Lambda}_1\cong\widehat{\Lambda}_2$ does not
contain a non-trivial finite normal subgroup.
 \end{proof}

Recall that the  {\em{profinite genus}} of a  finitely generated group $G$ consists of the set of isomorphism classes of finitely
generated, residually finite groups $H$ such that $\widehat{H}\cong\widehat{G}$.

In the following statement, we use Mostow-Prasad rigidity to conjugate the representatives
of the isomorphism classes in the genus of $\Delta$ so that they contain $\G$ itself rather an isomorphic copy.
The statement about normality is immediate from Proposition \ref{correspondence}. 

\begin{corollary}\label{c:only-latt}
Let $\G<\Delta< \PSL(2,\C)$ be lattices. If $\G$  is profinitely rigid, then the profinite genus of $\Delta$
is a subset of the lattices in $\PSL(2,\C)$ that contain $\G$ with index $[\Delta:\G]$. If $\G$ is normal in
$\Delta$, then it is normal in all of the lattices in the profinite genus of $\Delta$.
\end{corollary}

It follows from  Corollary \ref{c:only-latt} that if $\Delta$ is the only lattice of a given covolume containing $\G$,  
up to conjugacy, then $\Delta$ is profinitely rigid. 
Further rigid lattices can be obtained by using Corollary \ref{c:only-latt}  in combination with other profinite
invariants: for example, by goodness, one can distinguish lattices that contain torsion from those that are torsion-free.
The fact that normality is preserved in the standard correspondence of Proposition \ref{correspondence} can also be used in this context.

\begin{theorem}\label{normaliser} 
If a lattice in $\PSL(2,\C)$ is profinitely rigid (in the absolute sense), then so is its normaliser in $\PSL(2,\C)$.
\end{theorem}

\begin{proof} By Mostow-Prasad rigidity, the normaliser of a lattice $\G< \PSL(2,\C)$ can be characterised as the
unique lattice $\Omega< \PSL(2,\C)$ of minimal co-volume containing $\G$ as a normal subgroup (necessarily of finite index). 

Let $\Lambda$ be a finitely generated, residually finite group  with $\widehat{\Lambda}\cong\widehat{\Omega}$. 
Then, as in Corollary \ref{c:only-latt}, $\Lambda$ is isomorphic to a lattice with the same covolume
as $\Omega$ that contains a normal subgroup $\G_0$ isomorphic to $\G$. By conjugating the image of
$\Lambda$ in $\PSL(2,\C)$ if necessary, we may assume $\G_0=\G$. Then, by the characterisation of $\Omega$  in the 
first sentence of the proof,  $\Lambda=\Omega$.
\end{proof}

\subsection{Arithmetic lattices of simplest type}

For the most part, our focus in this article is on hyperbolic orbifolds of dimension 3. This reflects the special role that
arithmetic plays in dimension 3, in particular the existence and importance of the invariant trace field,
and the fact that lattices in ${\rm{PSL}}(2,\C)$ are good. But there are also lattices in higher dimensional hyperbolic
lattices that fit well with the circle of ideas developed here,  
namely {\em standard} arithmetic lattices, which are also known as arithmetic lattices of {\em simplest type} -- their definition is recalled below.

Let $k$ be a totally real number field of degree $d$ over $\Q$ equipped with a fixed embedding into
$\R$ which we refer to as the identity embedding, and denote the ring of integers of $k$ by $R_k$.  Let $V$ be an $(n+1)$-dimensional vector space over $k$ equipped with a non-degenerate quadratic form $\mathrm{f}$ defined over $k$ which has signature $(n,1)$ at the identity embedding, and signature $(n+1,0)$ at the remaining $d-1$ embeddings.

Given this, the quadratic form $\mathrm{f}$ is equivalent over $\R$ to the quadratic form $J_n=x_0^2+x_1^2+\ldots +x_{n-1}^2-x_n^2$, and for any non-identity Galois embedding $\sigma:k\rightarrow \R$, the quadratic form $\mathrm{f}^\sigma$ (obtained by applying $\sigma$ to each entry of $\mathrm{f}$) is equivalent over $\R$  to $x_0^2+x_1^2+\ldots +x_{n-1}^2+x_n^2$. We call such a quadratic form {\em admissible}.

Let $F$ be the symmetric matrix associated to the quadratic form $\mathrm{f}$ and let $\Or(f)$ and $\SO(f)$ denote
the linear algebraic groups defined over $k$ described as:
$$\Or(f)=\{X\in\GL(n+1,\C):X^tFX=F\}~\hbox{and}~                                                                                     
\SO(f)=\{X\in\SL(n+1,\C):X^tFX=F\}.$$

\noindent For a subring $L\subset \C$, we denote the $L$-points of $\Or(f)$ (resp. $\SO(f)$) by $\Or(f,L)$ (resp. $\SO(f,L)$).

Note that, given an admissible quadratic form defined over $k$ of signature $(n,1)$, there exists $T\in \GL(n+1,\R)$ such that $T^{-1}\SO(f,\R)T = \SO(n,1)$.  Let ${\rm{Isom}}^+(\H^n)$ denote 
the full group of orientation-preserving isometries of $\H^n$. This can be identified with the group $\SO^+(J_n,\R) = \SO^+(n,1)$, which is the subgroup of $\SO(n,1)$ preserving the upper-half sheet of the hyperboloid $J_n=-1$.

A subgroup $\Gamma < {\rm{Isom}}^+(\H^n)$ is called {\em arithmetic of simplest type} (alternatively, {\em a standard arithmetic lattice})  
 if $\Gamma$ is commensurable with the image in ${\rm{Isom}}^+(\H^n)$ of the subgroup of $\SO(f,R_k)$ (under the conjugation map described above). An arithmetic hyperbolic $n$-manifold $M={\H}^n/\Gamma$ is called arithmetic of simplest type if $\Gamma$ is.

The following analogue of Corollary \ref{c:only-latt} holds for arithmetic lattices of simplest type. 

\begin{theorem}
\label{simplest_type}  Let $n\ge 2$, 
and let $\G<{\rm{Isom}}^+(\H^n)$ be an arithmetic lattice of simplest type which is profinitely rigid. Suppose that 
 $\Delta< {\rm{Isom}}^+(\H^n)$ contains $\G$ as a subgroup of finite index. 
Then the profinite genus of $\Delta$
is a subset of the lattices in ${\rm{Isom}}^+(\H^n)$ that contain $\G$ with index $[\Delta:\G]$. If $\G$ is normal in
$\Delta$, then it is normal in all of the lattices in the profinite genus of $\Delta$.
\end{theorem}

\begin{proof} The proof of Corollary \ref{c:only-latt} will apply once we have established that the lattices we are studying
do not contain non-trivial finite, normal subgroups. To see that this is the case, we appeal to Proposition \ref{p:no-normal}.
Subgroups of ${\rm{Isom}}^+(\H^n)$ have torsion-free subgroups of finite index, and hence discrete
subgroups have finite virtual cohomological dimension, so the only non-trivial point to address is the goodness of 
arithmetic lattices of simplest type. This is a consequence of \cite{BHW}, as explained in \cite[Theorem 6.5]{KRS}.

Proposition \ref{p:no-normal} implies that if $\wh{\Delta}$ had a non-trivial finite normal subgroup, then 
$\Delta$ would too. But $\Delta$ has no such subgroup, because if it did then, as in the proof of Corollary \ref{c:SL-no-N}
there would be a $\Delta$-invariant totally geodesic submanifold of positive codimension in $\H^n$ (the fixed point set
of the finite normal subgroup) and this is impossible since $\Delta$ is a lattice.
\end{proof} 
 
We should offset the preceding discussion by noting that it is unknown
whether there exist profinitely rigid lattices in ${\rm{Isom}}^+(\H^n)$ for $n>3$.

\subsection{Flexibility in rigid sandwiches}

Theorem \ref{normaliser} and the examples in Section \ref{s:new} 
illustrate how the foregoing results can be used to generate pairs of lattices $\G < \Omega$
both of which are profinitely rigid. 
Ultimately, one expects all  intermediate lattices  
$\G<\Delta<\Omega$ to be profinitely rigid as well, but the following construction indicates that this is less obvious
than one might naively imagine.

It is easy to see that if $N$ is finite then $N\times\Z$ is profinitely rigid.
Gilbert Baumslag \cite{gilbert} pointed out that, strikingly, this rigidity 
fails if one replaces $N\times\Z$ with a semidirect product. We shall modify his construction to exhibit pairs of
non-isomorphic groups  $H_1$ and $H_2$ that have the same finite quotients and have the same finite
index in a group $N\times\Z$ with $N$ finite. 
Thus $H_1$ is sandwiched between the profinitely rigid groups $\Z$ and $N\times\Z$ but is not itself profinitely rigid.

Our construction applies (mutatis mutandis) to each of the pairs of metacyclic groups considered by Baumslag in 
\cite{gilbert}, but for the sake of clarity we concentrate on 
$$
B_1 = (\Z/25)\rtimes_\alpha \Z \ \ \ \ \ \ \ B_2 = (\Z/25)\rtimes_\beta \Z
$$
where, in multiplicative notation, $\alpha\in{\rm{Aut}}(\Z/25)$ is $\alpha(x)=x^6$ and $\beta(x)=x^{11}$. Note
that $\alpha = \beta^3$ and $\beta=\alpha^2$ generate the same cyclic subgroup of order $5$ in ${\rm{Aut}}(\Z/25)$.
One can prove by direct argument that $B_1\not\cong B_2$ but $\wh{B}_1\cong\wh{B}_2$. 

Let $N = (\Z/25)\rtimes_\alpha (\Z/5)=\<x,y\>$, where $x$ generates the first factor and the generator $y$ in the second
factor acts as $\alpha$. Let $t$ be a generator for the second factor of $N\times \Z$. Then
$H_1 = \< x, yt\>$ and $H_2 = \< x, y^2t\>$ both have index $5$ in $N\times\Z$, and $H_1\cong B_1$ while $H_2\cong B_2$.

By replacing $B_1$ and $B_2$ in the above construction by the surface-bundle groups of Hempel \cite{hempel}, one
can arrange for all the groups in the sandwich to be fundamental groups of closed 3-orbifolds.

\section{Kleinian groups: Theorems \ref{t1}  and \ref{t2}}
\label{klein}

In this section, we complete the proofs of Theorems \ref{t1} and \ref{t2} 
and then reflect on how these results can be used simplify arguments in \cite{BMRS1}.\\[\baselineskip]
\noindent{\bf Proof of Theorem \ref{morefinite1}:}~Let $\G <{\rm{PSL}}(2,\C)$ be a finitely generated, non-elementary Kleinian group
and let $H<\G$ be a subgroup of finite index $d>1$. If every finite quotient of $H$ is also a quotient of $\G$,
then  Propositions \ref{p1} and \ref{p2}  yield an epimorphism $\wh{\G}\to \wh{H}$ with finite kernel.
Corollary \ref{c:SL-no-N} tells us that $\G$ does not contain a non-trivial, finite normal subgroup, so in fact $\wh{\G}\cong \wh{H}$. 

Proposition \ref{correspondence} then provides an index-$d$ subgroup $H_1<H$, corresponding to $H<\G$, 
with $\wh{H}_1\cong \wh{H} \cong \wh{\G}$.   Repeating this argument, we obtain an infinite sequence
of subgroups $H_n<\G$ with $[\G:H_n]=d^n$ such that $\wh{H}_n\cong\wh{\G}$. 

If $\G$ is torsion-free and has finite co-volume, then this 
contradicts Liu's result \cite{Liu} that there can be only finitely many finite-volume hyperbolic $3$-manifolds whose fundamental groups 
have the same profinite completion.  (The covolume of $H_n$ increases with $n$, so no two of the manifolds $\H^3/H_n$ are homeomorphic,
by Mostow-Prasad rigidity.)
This completes the proof in the torsion-free case.

If $\G$ is a lattice with torsion, then we pass to a torsion-free subgroup of finite index $\Delta<\G$ and consider
the subgroups $\Delta_n<H_n$ obtained from the isomorphism $\widehat{\G}\cong\wh{H}_n$
via the correspondence of Proposition \ref{correspondence}. Then $\wh{\Delta}_n\cong \wh{\D}$ for all $n$.
As $[H_n:\Delta_n]= [\G:\Delta]$, the covolume of the lattices $\Delta_n$ increases with $n$.  
Lemma \ref{l:good-tf} assures us that $\Delta_n \hookrightarrow \wh{\Delta}$ is torsion-free. Thus we once again
obtain a contradiction to Liu's Theorem.

Proposition \ref{p:inf-vol} deals with the case where $\G$ has infinite co-volume.
\qed
 
 \bigskip

\noindent{\bf Proof of Theorem \ref{main1}:}~ We must show that if $\G<{\rm{PSL}}(2,\C)$ is a lattice and $H<\G$
is a proper subgroup, then $\wh{\G}$ is not isomorphic to $\wh{H}$. In the light of Theorem \ref{morefinite1} 
we may assume that $H$ is of infinite index. If a finitely generated, residually finite  group has the same profinite 
completion as a finitely generated virtually abelian group, then it is itself virtually abelian; this distinguishes $\G$ from
its elementary subgroups. When $H$ is non-elementary of
infinite index,  Lemma \ref{l:l2} tells us that $b_1^{(2)}(H)\neq 0$ while  $b_1^{(2)}(\G)= 0$. And $b_1^{(2)}$ 
 is an invariant of $\wh{\G}$ (Lemma \ref{l2}).  \qed
 
 \bigskip

\subsection{Reworking the end game in \cite{BMRS1}} \label{BMRS_redo} 
The proofs used in \cite{BMRS1} and \cite{BMRS2} to establish the profinite rigidity of 
arithmetic Fuchsian and Kleinian groups such as  the  triangle group $\Delta(3,3,4)$, the Bianchi group $\PSL(2,\Z[\omega])$, and the fundamental
group of the Weeks manifold, had two main steps. The first and deepest part of the proof 
uses Galois rigidity of these examples to deduce that if $\G$ is one of the above groups and
$\Lambda$ is any finitely generated, residually finite group with $\wh{\Lambda}\cong\wh{\G}$, then there exists a homomorphism
$\Lambda \rightarrow \G$ with non-elementary image, $L$ say.  A further 
suite of arguments of a more {\em ad hoc} nature was then used
in  \cite{BMRS1} and \cite{BMRS2}, exploiting specific characteristics of the finite index subgroups of $\G$  to show that $L=\G$. 
We then had an epimorphism $\wh{\G} = \wh{\Lambda}\to \wh{L}= \wh{\G}$, and the Hopfian 
property of  profinite groups (explained in the last paragraph of the proof of Proposition \ref{p2})
assures us that this is injective, hence $\Lambda\cong\G$.  The Galois rigidity of these examples is a consequence of a fine understanding of the arithmetic of the invariant
trace-field and quaternion algebra of the given lattice.

Theorem \ref{main1} provides a more uniform and conceptual alternative to this second suite of arguments, 
removing the need to analyse the case where $L$ has finite index in $\G$. 
(The case where $L$ has infinite
index is easy to deal with using topological arguments.) 

\begin{proposition}\label{p3}
Let $\G$ be a non-elementary Kleinian
group and let $\Lambda$ be a finitely generated, residually finite group such that there is a homomorphism $\Lambda\rightarrow \G$ whose image is a proper subgroup of finite index in $\G$.  Then $\wh{\Lambda}$ is not isomorphic
to $\wh{\G}$.
\end{proposition}

\section{New examples of profinitely rigid Kleinian groups}
\label{s:new}

We closed the previous section by discussing how Theorems \ref{t1} and \ref{t2} allow one to avoid  significant difficulties
in the second stage of the proofs  of profinite rigidity  in \cite{BMRS1} and \cite{BMRS2}. This raises the
hope that in other cases where one can prove Galois rigidity, one might be able to use Proposition \ref{p3} to deduce 
profinite rigidity. In this section we shall illustrate the potential of this approach by exhibiting 
new examples of profinitely rigid Kleinian groups. We shall also illustrate the potential of the observations
in Section \ref{s:go-up} 
by identifying further profinitely rigid groups among over-lattices of our first example. Along the way, we shall
describe an infinite family of  Galois rigid Kleinian groups. 

Let $K\subset S^3$ denote the figure-eight knot, let  $Q_n$ be the orbifold obtained by 
$(n,0)$-Dehn filling on $K$ and let $M_n$ be the $n$-fold cyclic branched cover of $K$.
Note that $M_n$ can also be regarded as an $n$-fold cyclic (orbifold) cover of $Q_n$; this is
the maximal abelian cover.
 When $n=2$, $M_2$ is the Lens Space $L(5,3)$, and so $Q_2$ is finite; when $n=3$, $M_3$ is the flat manifold known as the Hantzsche-Wendt
 manifold and so $Q_3$ is a Euclidean crystallographic group; when $n\geq 4$, $M_n$ and $Q_n$ are hyperbolic 
(see \cite{CHK} for example).  In addition, it was shown in \cite{HKM} that $\Gamma_n=\pi_1(M_n) \cong F(2,2n)$ where $F(2,2n)$ is one of the
Fibonacci groups of \cite{Co} (a presentation for $F(2,8)$ is given below). Let $\Delta_n$ be the orbifold fundamental group of $Q_n$. Then
 $Q_n=\H^3/\Delta_n$ and $\G_n=[\Delta_n,\Delta_n]$ is a subgroup of index $n$.   

The following theorem will be proved in the next section. 

\begin{theorem}
\label{new}
In the notation established above,
\begin{enumerate}
\item $\Delta_p$ is Galois rigid for all primes $p\geq 5$;
\item $\Delta_4, \Delta_6$ and $\Delta_9$ are Galois rigid;
\item $\Gamma_4$ is Galois rigid. 
\end{enumerate}
\end{theorem}

Taking Theorem \ref{new}(3) as our starting point, we shall construct a new family of profinitely rigid groups.

\begin{theorem}\label{t:G4} $\Gamma_4$ is profinitely rigid.
\end{theorem} 

Our proof  relies on the following additional information concerning $\G_4$. It is proved in \cite[\S 4.1]{Re0} that $\G_4$ is an arithmetic Kleinian group with (invariant) trace-field $k=\Q(\sqrt{-3})$ 
and that its invariant quaternion algebra $B$ over $k$  is ramified at the places $\nu_2$ and $\nu_3$ associated to the prime ideals of norm $4$
and $3$, respectively, in $k$. (The arithmeticity of $\Delta_4$ was also proved in \cite{HLM}.)  Note that these are the unique primes of these
norms in $k$. 
This implies that $B$ is {\em locally uniform} in the sense of \cite[Definition 4.10]{BMRS1}. Using volume considerations, it is also shown in
\cite[pp. 172--173]{Re0} that $\Delta_4^{(2)}=\Gamma_{\mathcal{O}}^1$, the image in $\PSL(2,\C)$
of the group of elements of norm one in the unique (up to $B^*$-conjugacy) maximal order $\mathcal{O} \subset B$ (that is to say, $B$ has type number $1$) 
It is clear from the construction of $Q_n$ that $\Delta_4^{\ab}\cong \Z/4\Z$, 
so $\Gamma_4=[ \Delta_4, \Delta_4] \subset \Delta_4^{(2)}$ has index $2$. 
It was shown in \cite{HKM} that $\Gamma_n$ is isomorphic to the Fibonacci group $F(2,2n)$ from \cite{Co}, and 
the abelianization of  $\G_4$  can be calculated from the presentation of $F(2,8)$
given below, $\G_4^{\rm{ab}}=H_1(M_4,\Z)\cong \Z/3\Z \oplus  \Z/15\Z$.  
\medskip

\begin{proof}
Suppose that $\Lambda$ is a finitely generated residually finite group with $\wh{\G}\cong\wh{\Lambda}$. Using Theorem \ref{new}(3), local
uniformity and,  
the uniqueness of the maximal order $\mathcal{O}$ (as noted above), together with \cite[Corollary 4.10]{BMRS1}, we can build a representation 
$\phi :\Lambda\rightarrow \Gamma_{\mathcal{O}}^1$, with non-elementary image $L$ say. We claim that $L\subset \Gamma_4$. This follows from
the calculation of $\G_4^{\rm{ab}} $ and the fact that
 $[\Gamma_{\mathcal{O}}^1:\Gamma_4]=2$: if $L$ were not contained 
in $ \G_4$, then $L\cap \G_4$ would have index $2$ in $L$,  but since $\G_4^{\rm{ab}}$ cannot map onto $\Z/2\Z$,
neither can $\G_4$ or $L$.
As $\G_4$ is torsion-free, if $L$ were of infinite index in $\G_4$ then it would be the fundamental group of a 3-manifold whose compact core had
a boundary component of positive genus, and  it would follow by duality (``half lives, half dies") that $L$ had infinite abelianization. But
$\G_4$ has finite abelianization, so $\Lambda$ (hence $L$) does too. Thus we may assume that $L$ has finite index in $\G_4$, in which
case Proposition \ref{p3} applies to show $L=\G_4$. The argument is completed, as in Section \ref{BMRS_redo}, by appealing to
the Hopfian property for profinite groups.
\end{proof}

\subsection{Some lattices containing $\G_4$}
The results in this section illustrate how the ideas presented in Section \ref{s:go-up} can be used in practice.
A discussion of the lattices that contain  $\G_4$ as a normal subgroup will suffice to make this point but a wider analysis covering
all of the lattices containing $\G_4$ would be feasible.

Borel \cite{Borel}  classified the maximal lattices in the commensurability classes of arithmetic lattices. In the case 
of $\G_4$,  as noted above, the invariant quaternion algebra $B$ has type number $1$, and so there is a unique lattice of minimal co-volume in its commensurability class, up to conjugacy.
Borel's volume formula shows that this lattice $\G_{\emptyset,\emptyset}$ has volume $v_0/8$,
where $v_0$ is the volume of the regular ideal simplex in $\H^3$. Jones and Reid \cite{JR}
obtain the following presentation,
$$
\G_{\emptyset,\emptyset}=\langle x, y, z \mid x^2 = y^2 = z^2 = (xyz)^4 = (xyxyxz)^2 = (yzxz)^2 = 1 \rangle.
$$

The arguments on pages 172--176 of \cite{Re0} identify $\G_{\emptyset,\emptyset}$ as the normalizer in ${\rm{PSL}}(2,\C)$
of $\Gamma_{\mathcal{O}}^1$ and we shall see in a moment that $\G_{\emptyset,\emptyset}$ is also the normalizer of
$\G_4$. Thus we have a chain of normal subgroups
$$
\G_4 < \Delta_4^{(2)} = \Gamma_{\mathcal{O}}^1 < \Delta_4 < \G_{\emptyset,\emptyset}
$$
where the first two are index-$2$ and the last is index-$4$. The
other lattices $\Omega$ with $\G_4<\Omega<\G_{\emptyset,\emptyset}$ correspond to subgroups in
$\G_{\emptyset,\emptyset}/\G_4$, which is a group of order $16$.

It is shown in \cite[pp. 172-173]{Re0} that $M_4$ double covers a manifold $M$ which is now known as $\Vol(3)$.
This is the only manifold that $M_4$ covers non-trivially.
The name $\Vol(3)$ was given to 
reflect the fact that this manifold had the third smallest known volume of a closed hyperbolic 3-manifold (namely $v_0$).  As noted in Section \ref{intro}, it has recently been proved \cite{GHMTY} that $\Vol(3)$ is the unique closed orientable
hyperbolic $3$-manifold of volume $v_0$. (It was previously known that $\Vol(3)$ was 
the only closed orientable arithmetic hyperbolic 3-manifold with volume is $v_0$.)
In particular, up to conjugacy, $\vv_3:=\pi_1(\Vol(3))$ is the only
torsion-free lattice among the $\Omega$ described above. It is proved in \cite{JR} that $\vv_3$ is normal in $\G_{\emptyset,\emptyset}$ with quotient a dihedral group of order $8$.
Moreover $H_1(\Vol(3),\Z) = \Z/3\Z\times \Z/6\Z$, so $\vv_3$ has a unique subgroup of index $2$, which is $\G_4$.
As $\G_4$ is characteristic in $\vv_3$, it too is normal in $\G_{\emptyset,\emptyset}$. And since $\G_{\emptyset,\emptyset}$ is a maximal subgroup $\PSL(2,\C)$, 
this is precisely the normalizer of $\G_4$ in $\PSL(2,\C)$. 

\begin{corollary}\label{c:up-from-G4}
The groups $\Gamma_{\mathcal{O}}^1,\ \Delta_4,\ \vv_3$ and $\G_{\emptyset,\emptyset}$  are all
profinitely rigid in the absolute sense.
\end{corollary}

\begin{proof} Non-elementary representations of $\Gamma_{\mathcal{O}}^1$ restrict to non-elementary representations of 
$\G_4$ and non-conjugate ones remain non-conjugate. Therefore $\Gamma_{\mathcal{O}}^1=\Delta_4^{(2)}$ is Galois
rigid and the proof of Theorem \ref{t:G4} shows that it is profinitely rigid.

It follows from Corollary \ref{c:only-latt} that in order to prove that $\vv_3$ is profinitely rigid, we need only distinguish it
from lattices of the same co-volume that contain $\G_4$ as a subgroup of index $2$.
By Corollary \ref{c:good-SL} we can also restrict to lattices that are torsion-free. And, as we noted above, $\vv_3$
is the unique such lattice. 

As $\G_{\emptyset,\emptyset}$ is the unique lattice of minimal co-volume containing $\G_4$, Corollary \ref{c:only-latt}
tells us that it too is profinitely rigid.   

Using Magma \cite{Mag} to enumerate the subgroups of index $4$ in $\G_{\emptyset,\emptyset}$, one finds that there are
$11$ conjugacy classes, only one of which has abelianization $\Z/4\Z$, so this must be $\Delta_4$. 
As $\G_{\emptyset,\emptyset}$ is the normalizer of $\Gamma_{\mathcal{O}}^1=\Delta_4^{(2)}$, it follows
that, up to conjugacy, $\Delta_4$ is the only lattice that contains $\Gamma_{\mathcal{O}}^1$ as a subgroup
of index $2$ and has abelianization $\Z/4\Z$. Thus  $\Delta_4$ is profinitely rigid, by Corollary \ref{c:only-latt}. 
\end{proof}

With further computation one could extend the above analysis to the other $\Omega$ with $\G_4<\Omega<\G_{\emptyset,\emptyset}$.
We mention one other example in order to illustrate a further point concerning the ideas in Section \ref{s:go-up},
namely that one can use $\Delta/\Lambda$ as an invariant when $\Lambda<\Delta$ is normal and profinitely rigid. In the
current setting, as $\G_{\emptyset,\emptyset}/\vv_3$ is dihedral of order 8, there is a unique lattice $\Omega_4<\G_{\emptyset,\emptyset}$
that contains $\vv_3$ as a normal subgroup such that the quotient is cyclic of order $4$. From Corollary \ref{c:only-latt} and
the last item in Proposition \ref{correspondence} we deduce:

\begin{corollary}\label{c:Omega4}
$\Omega_4$ is profinitely rigid.
\end{corollary}

\section{Galois rigidity for cyclic branched covers of the figure-eight knot}\label{s:galois}

Our proof of Galois rigidity for $\Delta_n$ (Theorem \ref{new}) 
is based on the understanding of the ${\rm{PSL}}(2,\C)$ representations of the
fundamental group of the figure-eight knot complement
$$\pi_1(S^3\setminus K) = \<\a,\b \mid (\a\b^{-1}\a^{-1}\b)\a(\b^{-1}\a\b\a^{-1}=\b\>.$$ 
In this presentation, the (conjugate) generators $\a$ and $\b$ represent meridians of the knot, so $\D_n=\pi_1^{\rm{orb}}(Q_n)$
is obtained by adding the relation $\a^n=1$. It is easy to check that $\D_2$ is the dihedral group of order $8$ while
$\D_3$ is a virtually abelian group.

\subsection{$\Delta_n$ is Galois rigid if $n\ge 5$ is prime}  
Although we wish to understand $\PSL(2,\C)$ representations, it will be convenient to work with $\SL(2,\C)$ representations.  
To that end, 
consider the  canonical component $X_0$ of the $\SL(2,\C)$ character variety for $\pi_1(S^3\setminus K)$; this is
the component containing the characters of all irreducible $\SL(2,\C)$ representations of $\pi_1(S^3\setminus K))$. It 
can be identified with the vanishing set of the polynomial
$$P(T,R)=1 + R - R^2 - 2T^2 + RT^2,$$
where $T=\chi_\rho(\a)=\chi_\rho(\b)$ and $R=\chi_\rho(\a\b)$ (see \cite[Section 7]{CRS} for example).  
The character of each irreducible representation of $\Delta_n$ corresponds to a point on $X_0$ where
the  relation $\a^n=1$ holds in $\PSL(2,\C)$, i.e. where $T$ specializes to  
$T_{n,k}=2\cos(k\pi/n)$ with $k\in\{1,\dots,n-1\}$. Recall that when $n$ is prime $\Q(\cos\pi/n)$ is a Galois extension of $\Q$ of degree $(n-1)/2$, and moreover given $k\in \{1,\dots,n-1\}$ there is exactly one $j$ such that 
$2\cos(j\pi/n)=-2\cos(k\pi/n)$. 

Each of these specializations gives rise to a character determined 
by solving for $R$ in 
$$p_{n,k}(R) = 1 + R - R^2 - 2T_{n,k}^2 + RT_{n,k}^2.$$ 
Thus $\Delta_n$ has (at most) $2(n-1)$
irreducible $\SL(2,\C)$
representations, up to conjugacy. Furthermore, the preceding comment about signs implies that the set 
of characters for these representations  can be divided up so that each character in one half is the negative of 
one in the other half. Thus, on projecting  
to $\PSL(2,\C)$, we have exactly $n-1$ representations, up to conjugacy. Hence we will be done if we can argue that the trace-field of $\Delta_n$
has degree $(n-1)$ (since $n\geq 5$ is prime the trace-field coincides with the invariant trace-field).

To see that this is the case, we consider the character of the discrete faithful representation:
here, $T$ specializes to $T_{1,n}= 2\cos(\pi/n)$ and we note that $p_{n,1}(R)$ must be irreducible over $\Q(\cos(\pi/n))$,
for if not then the trace field $\Q(\tr(\Delta_n)) = \Q(T_{1,n},R_n)=\Q(T_{1,n},\-{R}_n)$ would be totally real, which it
is not since $Q_n$ is a closed hyperbolic $3$-orbifold. (Here, we have denoted the roots of $p_{1,n}(R)$ by $R_n,\-{R}_n$.) 
Thus the trace field of $\Delta_n$ is a number field of degree $(n-1)$ and so $\Delta_n$ is Galois rigid.
(The  representations corresponding to characters arising from the specializations $T=T_{k,n}$ with $k>1$ are appearing 
as Galois conjugates of the discrete faithful representation.) \qed
\medskip

\subsection{$\Delta_n$ is Galois rigid if $n\in\{4,6,9\}$}~We noted earlier that 
$\pi_1(S^3\setminus K)$ becomes finite if one imposes the relation $\a^2=1$ and virtually abelian if one imposes $\a^3=1$.
Thus, for $n\in\{4,6,9\}$, in any irreducible representation of $\Delta_n$, the image of $\alpha$ must have order at least $4$. With this additional 
information, the preceding proof applies upon specializing $T$ in the following way: $\pm 2\sqrt{2}$ (when $n=4$), $\pm 2\sqrt{3}$ (when $n=6$) and $\pm 2\cos(\pi/9)$, $\pm 2\cos(2\pi/9)$, $\pm 2\cos(4\pi/9)$ ($n=9$).
\qed
\medskip
\begin{remark} Note Galois rigidity fails for $n=8$ since $\Delta_8$ admits an epimorphism onto  $\Delta_4$. \end{remark}

\bigskip

\subsection{$\G_4$ is Galois rigid}~This proof is of a different nature to the one above. Instead of locating representations within the well-understood character variety of
a group mapping onto $\G_4$,  we take a more brute force approach, analysing directly what the irreducible representations of 
$\G_4$ can be, starting from a specific finite presentation and making explicit 
computations in Mathematica \cite{Math}, as described below. The Mathematica notebook is available from the authors upon request. The
presentation that we work with comes from the isomorphism $\G_4\cong F(2,8)$ with a Fibonacci group,
$$\<x_1,\dots,x_8 \mid  x_1x_2x_3^{-1}, x_2x_3x_4^{-1}, \dots, x_6x_7x_8^{-1}, x_7x_8x_1^{-1}, x_8x_1x_2^{-1}\>,$$ 
which reduces to the following two generator presentation: 
\begin{equation}
\label{eq1}
\G_4=\<a,b \mid b a^{-2} b a^{-1} b^2 a b^2 a^{-1},    a^2 b a b^2 a b a^2 b^{-1}\>
\end{equation}
where $a=x_3$ and $b=x_4$. One can check directly (by hand or computer)
that imposing either of the additional relations $a^2=1$ or $a^3=1$
yields a finite quotient, and this will be useful in our analysis of 
the characters of irreducible $\SL(2,\C)$-representations of $\G_4$. 
Before we begin, we note that $H^2(\G_4,\Z/2)=H_1(\G_4,\Z/2)=0$ by Poincar\'{e} duality, so the standard discrete faithful
representation $\G_4\to \PSL(2,\C)$ and its complex conjugate lift to $\SL(2,\C)$, and so as above, it is convenient to work in 
$\SL$ rather than $\PSL$.

Since we are only interested in irreducible representations, 
we can assume that if $\rho:\G_4\rightarrow \SL(2,\C)$ is such a representation then $\rho(\G_4)$ is not conjugate into the subgroup consisting of upper triangular matrices in $\SL(2,\C)$. Standard considerations ensure that we can conjugate 
$\rho$ so that $\rho(a)$ fixes $\infty$ and $\rho(b)$ fixes $0$; i.e.
\[ \rho(a) = \begin{pmatrix} x & 1\cr 0& 1/x \cr\end{pmatrix}, ~~ \rho(b) = \begin{pmatrix} y & 0\cr r & 1/y \cr\end{pmatrix}. \]

Referring to (\ref{eq1}), to handle evaluation in Mathematica of the first relation on the matrices, we split it up as follows: {\tt{rel1=Factor[w1-Inverse[w2]]}},
where $w1=ba^{-2}ba^{-1}b^2$ and $w2=ab^2a^{-1}$.   This results in a matrix $R$ whose entries are shown in \S \ref{math} and which solves to zero.
Using the $(1,2)$-entry of $R$ we can solve for $r$ in terms of $x$ and $y$, namely
$$r=\frac{x^4+x^3 y^4-x^3+x^2+y^2}{x y \left(x^2-xy^2-x+1\right)}.$$
We then re-evaluate the first relation in (\ref{eq1}) using this expression for $r$. 
This gives the matrix $R_1$ shown in \S \ref{math} solving to zero. 
To evaluate the second relation of (\ref{eq1}) on matrices, we again split this up: {\tt{rel2=Factor[w3-Inverse[w4]]}}, where $w3=a^2bab^2$ and $w4=aba^2b^{-1}$.  Using the value of $r$ displayed above, results in a matrix $S$
which solves to zero. The numerators $s_i$ $i=1,\ldots ,4$ of the entries of $S$ are displayed in \S \ref{math}.

As discussed in \S \ref{math}, an analysis using Mathematica shows that the only possible solution for $x$ arises from roots of the polynomial $p(x)=x^4-x^3+3 x^2-x+1=0$. Note that $p(x)$ is a symmetric polynomial, so that setting $x_0$ to be a root of $p(x)$, the other roots are $1/x_0$, $\overline{x}_0$, and 
$1/\overline{x}_0$.  Converting to characters shows that $\chi_\rho(a)$ is a root of the equation $X^2-X+1=0$. A similar analysis shows that this polynomial also provides the only possibilities for $y$, and we see  that $\chi_\rho(b)=\chi_\rho(a)$
or $\chi_\rho(b)=\overline{\chi}_\rho(a)$. It is also shown in \S \ref{math} that
$r$ satisfies a degree $4$ polynomial.  

As discussed in \S \ref{math}, the character of the faithful discrete representation $\rho_0$ of $\G_4$ satisfies $\chi_{\rho_0}(b)=\overline{\chi}_{\rho_0}(a)$, and 
$\chi_{\rho_0}(ab)=\chi_{\rho_0}(a)$,  and that the only other distinct character of an irreducible representation of $\G_4$ is that obtained by applying complex conjugation to $\rho_0$.  It follows that $\G_4$ is Galois rigid as claimed.\qed

\medskip 

\subsection{Mathematica output}
\label{math}

\bigskip

\noindent The matrix $R$:

\medskip

\noindent $(1,1)-\rm{entry} =
 \frac{y^6-r x^3 y^5-2 r x y^5+r^2 x^4 y^4+r^2 x^2 y^4-r x^5 y^3-r x^3 y^3+r x^2 y^3-r x
   y^3+r^2 x^4 y^2+r^2 x^2 y^2-r x^5 y-r x^3 y+r x^2 y-x^3}{x^3 y^2}$
    
\smallskip

\noindent $(1,2)-\rm{entry} = -\frac{x^4+y^4x^3-r y x^3-x^3+r y^3 x^2+r y x^2+x^2-r y x+y^2}{x^2 y^2} $
   
   \smallskip
   
\noindent $(2,1)-\rm{entry} =$

\smallskip

\noindent $\frac{r \left(y^2 x^6+x^6-2 r y^3 x^5-2 r y x^5+r^2 y^4 x^4+y^4 x^4+r^2 y^2 x^4-r y^5
   x^3-r y^3 x^3-r y x^3+r^2 y^4 x^2+r^2 y^2 x^2-2 r y^5 x+y^4 x-r y^3 x+y^2
   x+y^6\right)}{x^3 y^3}$
   
   \smallskip
   
\noindent $(2,2)-\rm{entry} = -\frac{x^2 y^6+r x y^5+r y^3+r x y^3-r^2 x^3 y^2-r^2 x y^2+2
   r x^4 y+r x^2 y-x^5}{x^2 y^4} $

\medskip

\noindent The matrix $R_1$ has $(1,2)$-entry $0$ and

\medskip

\noindent $(1,1)-\rm{entry} =\frac{\left(x^2+1\right) \left(x^2+x+1\right) \left(x y^4+x^2-x+1\right) \left(x^2 y^4-x
   y^4+y^4+2 x^2 y^2-3 x y^2+2 y^2+x^2-x+1\right)}{x^2 \left(x^2-y^2 x-x+1\right)^2} $
   
    \smallskip
   
\noindent $(2,1)-\rm{entry} =$
   
 \smallskip
   
\noindent $\frac{\left(x^2+1\right) \left(x^2 y^4-x y^4+y^4+2 x^2 y^2-3 x y^2+2 y^2+x^2-x+1\right)
   \left(x^4+y^4 x^3-x^3+x^2+y^2\right) \left(x^3 y^4+x^2 y^4+x y^4+x^3 y^2+x^2
   y^2+x^2-x+1\right)}{x^3 y^2 \left(x^2-y^2 x-x+1\right)^3} $
   
\smallskip
   
\noindent $(2,2)-\rm{entry} =  \frac{\left(x^2+1\right)
   \left(x^2+x+1\right) \left(x^2 y^4-x y^4+y^4+2 x^2 y^2-3 x y^2+2 y^2+x^2-x+1\right)}{x
   \left(x^2-y^2 x-x+1\right)^2}$

\medskip

\noindent 
We require $R_1$ to be the zero matrix in order to get a representation, and focussing on the $(2,2)$ entry will be
particularly informative: 
the factors $x^2+1$ and $x^2+x+1$ can be disregarded, because setting them equal to zero would imply that $a$ had order $2$ or $3$, and we noted earlier that $\G_4$ does not have an infinite quotient in which $a$ has order $2$ or $3$. Thus  
the $(2,2)$ entry tells us that $x^2 y^4-x y^4+y^4+2 x^2 y^2-3 x y^2+2 y^2+x^2-x+1=0$.

\medskip

\noindent The matrix $S$.

\medskip

\noindent $s_1 = x^{10} y^6+2 x^{10} y^4+2 x^{10} y^2+x^{10}-x^9 y^6-3 x^9 y^4-4 x^9 y^2-2 x^9+2 x^8 y^6+6
   x^8 y^4-x^8 y^3+8 x^8 y^2+5 x^8+x^7 y^8-x^7 y^7+x^7 y^6+x^7 y^5-3 x^7 y^4+2 x^7 y^3-9
   x^7 y^2-6 x^7+x^6 y^9-x^6 y^8+x^6 y^7-x^6 y^5+6 x^6 y^4-4 x^6 y^3+12 x^6 y^2+8 x^6+x^5
   y^8-2 x^5 y^7+2 x^5 y^6+2 x^5 y^5-3 x^5 y^4+4 x^5 y^3-10 x^5 y^2-6 x^5+x^4 y^9-x^4
   y^8+x^4 y^7+x^4 y^6-2 x^4 y^5+7 x^4 y^4-5 x^4 y^3+11 x^4 y^2+5 x^4+x^3 y^8+x^3 y^6+4
   x^3 y^5-4 x^3 y^4+4 x^3 y^3-6 x^3 y^2-2 x^3-x^2 y^7+2 x^2 y^6-4 x^2 y^5+5 x^2 y^4-4
   x^2 y^3+5 x^2 y^2+x^2+x y^7-x y^6+3 x y^5-2 x y^4+2 x y^3-x y^2+y^6-y^5+2 y^4-y^3+y^2$

\medskip

\noindent $s_2 = (y+1)(x^6 y^4-x^6 y^3+2 x^6 y^2-x^6 y+x^6-x^5 y^4+x^5 y^3-2 x^5 y^2+x^5 y-x^5+2 x^4y^4$
$-2 x^4 y^3+4 x^4 y^2-3 x^4 y+3 x^4-x^3 y^4+2 x^3 y^3-3 x^3 y^2+2 x^3 y-2 x^3+2 x^2y^4-2 x^2 y^3$
$+4 x^2 y^2-3 x^2 y+3 x^2-xy^4+x y^3-2 x y^2+x y-x+y^4-y^3+2 y^2-y+1)$
   
   \medskip
   
\noindent $s_3 =   -(x^4+x^3 y^4-x^3+x^2+y^2)(x^5 y^5+x^5 y^3-x^4 y^4-x^4 y^2+x^4 y-x^4+x^3y^5+x^3 y^2-x^3 y+x^3$
$-x^2 y^4+x^2 y^3-2 x^2 y^2+2 x^2 y-2 x^2+x y^4-x y^3+2 x y^2-xy+x-y^4+y^3-2 y^2+y-1)$
   
   \medskip
   
\noindent $s_4 =    x^9 y^7+2 x^9 y^5+x^9 y^3-x^8 y^7-2 x^8 y^5+x^8 y+2 x^7 y^7+3 x^7 y^5-x^7 y^3-2 x^7 y-x^6
   y^7+4 x^6 y^3+4 x^6 y+x^6+2 x^5 y^7+2 x^5 y^5+x^5 y^4-4 x^5 y^3-x^5 y^2-4 x^5 y-2
   x^5-x^4 y^7-x^4 y^6+x^4 y^5-x^4 y^4+5 x^4 y^3+x^4 y^2+4 x^4 y+4 x^4+x^3 y^7+x^3 y^4-3
   x^3 y^3-3 x^3 y^2-2 x^3 y-4 x^3+2 x^2 y^5+x^2 y^4+3 x^2 y^3+3 x^2 y^2+x^2 y+4 x^2-x
   y^5-x y^4-x y^3-3 x y^2-2 x+y^5+y^3+y^2+1$

\medskip

To determine the possibilities for $x$, $y$ and $r$ we use the {\tt{Resultant}} feature in Mathematica to eliminate variables and determine solutions.  
We have already noted that from the $(2,2)$ entry of $R_1$ we know that
$x^2 y^4-x y^4+y^4+2 x^2 y^2-3 x y^2+2 y^2+x^2-x+1$ must vanish.  On taking resultant with $s_i$, for $i=1, \ldots, 4$ to eliminate $y$ 
we are left with the polynomial $p(x)=x^4-x^3+3 x^2-x+1$.

On taking the resultant of $p(x)$ and $\left(x^2 y^4-xy^4+y^4+2 x^2 y^2-3 x y^2+2 y^2+x^2-x+1\right)$ to eliminate $x$ we obtain two possible polynomials: $\left(y^4-y^3+3 y^2-y+1\right)$ and $\left(y^4+y^3+3 y^2+y+1\right)$. Note the latter polynomial is the former evaluated at $-y$. This determines the same $\PSL(2,\C)$ representation, and so we deduce the polynomial that $y$ satisfies is also $p(x)$.

Returning to the matrix $R$, we can use the polynomials that were obtained for $x$ and $y$ and use resultants with each of the numerators of the entries of $R$ to determine $r$.  The only possible common factor is the polynomial 
$r^4-9r^2+36$.  Note this is not quadratic, but we know that for the faithful discrete representation $\rho_0$, characters $\chi_{\rho_0}(a)$, $\chi_{\rho_0}(b)$ and $\chi_{\rho_0}(ab)$ lie in $\Q(\sqrt{-3})$. This occurs when $y=\overline{x}_0$, where $x_0$ is (approximately) the root $0.14840294359835 - 0.632502179219i$
of $p(x)$ and the corresponding value of $r$ is (approximately) $r_0=-2.29128784747792+ 0.8660254037844386467i$. Setting $x=1/x_0$ and $y=1/\overline{x}_0$ with $r=r_0$ provides a representation with the same character as $\rho_0$
and hence determines a conjugate representation (see \cite[Proposition 1.5.2]{CS}). Finally we note that setting $x=y=x_0$ does not yield a representation (since no  value of $r$ solves the first relation of our presentation).

Thus we have shown that the only possible values of $x,y,r$ are those that give the discrete faithful representation and its conjugate, and
the proof that $\G_4$ is Galois rigid is complete.




\end{document}